\newif\ifels
\elsfalse

\ifels
\documentclass[preprint,sort&compress]{elsarticle}
\renewcommand{\cite}{\citep}
\else
\documentclass[11pt,a4paper]{article}
\fi

\usepackage[normalem]{ulem}
\usepackage{soul}
\usepackage{amsmath}
\usepackage{amsfonts}
\usepackage{amssymb}         %additional math symbols (see amsguide.tex)
\usepackage{amsopn}
\usepackage{theorem}          %additional LaTeX2e package (see pctex\latex2e)
\usepackage{latexsym}
\usepackage{graphicx}        %graphic and .eps files
\usepackage{mathrsfs}		%calligraphic fonts
\usepackage{psfrag}
\usepackage{algorithmic}
\usepackage{algorithm}
\usepackage{xcolor}
\usepackage{verbatim}
\usepackage{url}
\usepackage{enumerate}
\usepackage{wrapfig}
\usepackage{comment}
\usepackage{stackengine}

\definecolor{midgrey}{gray}{0.6}
\definecolor{darkgrey}{gray}{0.4}

\ifels
\else
\fi

\newtheorem{result}{Result}[section]
\newtheorem{thm}[result]{Theorem}
\newtheorem{defn}[result]{Definition}
\newtheorem{lem}[result]{Lemma}
\newtheorem{cor}[result]{Corollary}
\newtheorem{prop}[result]{Proposition}

\newenvironment{proof}
 {{\sl Proof.}\hspace*{1 ex}}%
 {{\nopagebreak\hspace*{\fill}$\Box$\par\vspace{12pt}}}

\ifels
\renewcommand{\qed}{}
\else
\newcommand{\qed}{\hfill \ensuremath{\Box}}
\fi

\newcommand{\transpose}[1]{{#1}^{\top}}

\newcommand{\xapprox}{\tilde{x}}
\newcommand{\val}[1]{\mathsf{v}(#1)}
\newcommand{\dly}{{z}}
\newcommand{\Kinf}{\stackanchor[1pt]{\tiny$\mathbb{E}$}{\tiny$\infty$}}
\DeclareMathOperator*{\argmin}{arg\,min}
\definecolor{plgreen}{rgb}{0,0.5,0}

\begin{document}

\ifels

\title{Random projections for conic programs\thanks{This paper has received funding from the European Union's Horizon 2020 research and innovation programme under the Marie Sklodowska-Curie grant agreement n.~764759 ``MINOA''.}}

\author[lix]{Leo Liberti}
\ead{liberti@lix.polytechnique.fr}

\author[riken]{Pierre-Louis Poirion}
\ead{pierre-louis.poirion@riken.jp}

\author[fpt]{Ky Vu}
\ead{vukhacky@gmail.com}

\address[lix]{LIX CNRS, Ecole Polytechnique, Institut Polytechnique de Paris, 91128 Palaiseau, France}

\address[riken]{RIKEN Center for Advanced Intelligence Project, Tokyo, Japan}

\address[fpt]{Dept.~of Mathematics, FPT University, Hanoi, Vietnam}

\else

\thispagestyle{empty}
\begin{center} 

{\LARGE Random projections for conic programs}
\par \bigskip
{\sc Leo Liberti${}^3$, Pierre-Louis Poirion${}^2$, Ky Vu${}^1$}
\par \bigskip
 {\small
   \begin{enumerate}
   \item {\it LIX CNRS, \'Ecole Polytechnique, Institut Polytechnique de Paris, F-91128 Palaiseau, France} \\ Email:\url{liberti@lix.polytechnique.fr}
   \item {\it RIKEN Center for Advanced Intelligence Project, Tokyo, Japan}
   \item {\it Dept.~of Mathematics, FPT University, Hanoi, Vietnam} 
   \end{enumerate}
 }
\par \medskip \today
\end{center}
\par \bigskip

\fi

% insert abstract
\begin{abstract}
We discuss the application of random projections to conic programming: notably linear, second-order and semidefinite programs. We prove general approximation results on feasibility and optimality using the framework of formally real Jordan algebras. We then discuss some computational experiments on randomly generated semidefinite programs in order to illustrate the practical applicability of our ideas.
\ifels
\else
\fi
\end{abstract}

\ifels
\maketitle
\fi

% insert paper
\section{Introduction}
\label{s:intro}
In this paper we study Random Projection (RP) techniques for Symmetric Conic Programming (SCP): the class of optimization problems over symmetric cones. This class includes several convex optimization classes, such as Linear Programming (LP), Second-Order Cone Programming (SOCP), and Semidefinite Programming (SDP). It turns out that a symmetric cone can be represented as the cone of squares of a formally real Jordan algebra (FRJA) \cite{jordanalg}. Using this framework it is possible to obtain theoretical results that apply to LP, SOCP and SDP alike \cite{alizadeh}. We employ this framework in order to derive approximate feasibility and optimality results related to the application of RPs to SCPs. We then exhibit some computational experiments to show that our techniques can be used to approximately solve large scale SDPs in acceptable amounts of time.

RPs are random matrices that reduce the dimensions of the vectors in a given set while approximately preserving all pairwise Euclidean distances with high probability (whp), which means that the probability of failure decreases exponentially fast with increasing size of the reduced dimension. The normal application setting of RPs is to numerical data. This can speed up algorithms (at some cost in terms of the solution quality) which are essentially based on Euclidean distances, such as, e.g., k-means or k-nearest neighbours \cite{boutsidis2010,indyk}. The success of the application of RP to these algorithms is not surprising insofar as RPs give approximation guarantees for Euclidean distances by definition. This work continues the sequence of works \cite{jll-dam,jllmor,ipco19,rpqp}, which support the much more counter-intuitive statement that a Mathematical Programming (MP) formulation might be approximately invariant (as regards feasibiliy and optimality) w.r.t.~randomly projecting the input parameters. In particular, in this work we generalize the results of \cite{jllmor} from the LP setting to all SCPs. Similarly to \cite{jllmor}, our results assume that: (a) feasible instances are explicitly bounded; (b) strong SCP duality holds. 

This work is motivated by the need to solve ever larger conic optimization problems. Its impact extends to LP, SOCP, SDP. Its importance is magnified by the fact that current SDP solution technology scales quite poorly in practice. 

As far as we know, this work provides the first theoretical analysis of the application of RPs to the general class of SCPs. Notwithstanding, there exist previous works about the application of RPs to specific types of symmetric cones. The paper \cite{jllmor}, for example, is about LPs. In \cite{bluhm2019}, the authors reduce the dimension of SDPs using the RP matrix mapping $M\mapsto TMT^\top$, where $T$ is a RP matrix. This approach is dual to ours, as we reduce the number of constraints rather than the dimension of the problem. We also note some literature results about fast low-rank approximations of PSD matrices using RPs \cite{tropp2,clarkson2}. Based on these works, \cite{grimmer} exploits RPs to propose a bundle method for approximately solving SDPs without ever storing the whole constraint matrix and right hand side vector. While a small fraction of the theoretical results we present here bear some similarity to those in \cite{tropp2,clarkson2,grimmer}, the context and proofs are different. The main difference is that the results in \cite{tropp2,clarkson2,grimmer} refer to the error yielded by some given algorithms deployed on randomly projected input. The results in this paper are about the formulation itself: we provide error guarantees on the optimality and feasibility of projected formulations w.r.t.~the original ones independently of the solution method. 

We note that even though this work generalizes the results of \cite{jllmor}, the proof techniques are completely different, and justifiably so. A cornerstone of the results in \cite{jllmor} is that $\transpose{T}Tx$ is approximately like $x$ for a random projection matrix $T$ (see below for a definition), whereas in the general conic case addressed in this paper, where $x$ is a PSD matrix, the corresponding entity would not even be symmetric: we sidestep this difficulty by using a completely different framework, i.e.~that of FRJA. Moreover, in \cite{jllmor} we extensively use the usual partial order on $\mathbb{R}^n$, which FRJA abstracts from by generalization to symmetric cones. Lastly, FRJA is essential in Sect.~\ref{s:erretr}, where we estimate the infeasibility of the retrieved solution.

A RP is a random matrix sampled from a sub-Gaussian distribution \cite{dirksen}. Among other properties, applying a RP to a vector in a given set yields a vector with approximately the same norm. One famous example of the application of RPs is the Johnson-Lindenstrauss lemma (JLL) \cite{jllemma}, which shows that a set of points can be embedded into a much lower dimension while keeping all pairwise distances approximately the same. 
\begin{defn}
\label{defn:rp}
  A random matrix $T \in \mathbb{R}^{d \times m}$ is called a {\it random projection} if there are universal constants $c$ and $\mathcal{C}$ such that, for any $x \in \mathbb{R}^m$ and $0 < \varepsilon < 1$ we have:
  \begin{equation*} 
    \mbox{\sf Prob} \big[ (1 - \varepsilon) \|x\|^2_2 \le \|Tx\|^2_2 \le (1 + \varepsilon) \|x\|^2_2\big] \ge 1-ce^{-\mathcal{C}\varepsilon^2 d}.
  \end{equation*}
\end{defn}
An easy corollary of the JLL is the preservation of scalar products. 
\begin{prop}
  \label{scalarJLL}
  For $T \in \mathbb{R}^{d \times m}$ a RP with $\varepsilon\in(0,1)$ and $x, y \in \mathbb{R}^m$, we have:%
  \begin{equation*} 
    \mathsf{Prob}\big[\,\langle x, y \rangle - \varepsilon \|x\|_2 \|y\|_2 \le 
    \langle Tx, Ty \rangle \le 	\langle x, y \rangle + \varepsilon \|x\|_2 \|y\|_2\,\big] \ge 1-2ce^{-\mathcal{C}\varepsilon^2 d}.
  \end{equation*}
\end{prop}
See \cite[Prop.~1]{jllmor} for a proof. We also derive a variant of this result for later use.
\begin{cor}
  \label{cor:scjll}
  For $T\in\mathbb{R}^{d\times m}$ a RP with $\varepsilon\in(0,1)$ and any $x\in\mathbb{R}^m$, we have:%
  \begin{equation*}
    \mathsf{Prob}\big[\|(\transpose{T}T x)_i - x_i\|_\infty \le \varepsilon \|x\|_2\big]\ge 1-2mce^{-\mathcal{C}\varepsilon^2 d}.
  \end{equation*}
\end{cor}
\begin{proof}
  For every $i\le m$, let $e_i=(0,\ldots,0,1_i,0,\ldots,0)\in\mathbb{R}^m$. From Prop.~\ref{scalarJLL} applied to $e_i$ and $x$, we have
  \[\transpose{e}_ix -\varepsilon\|x\|_2 \le \langle Te_i,Tx\rangle = \langle e_i,\transpose{T}Tx\rangle \le \transpose{e}_ix + \varepsilon \|x\|_2\]
for all $i\le m$ whp, whence $-\varepsilon\|x\|_2 \le \langle e_i,\transpose{T}Tx\rangle - \transpose{e}_ix \le \varepsilon \|x\|_2$ and $-\varepsilon\|x\|_2 \le (\transpose{T}Tx - x)_i \le \varepsilon \|x\|_2$. Thus, we have $|(\transpose{T}Tx-x)_i|\le\varepsilon\|x\|_2$ for any $i\le m$. A union bound argument over all $i\le m$ yields $\|\transpose{T}Tx-x\|_\infty\le\varepsilon\|x\|_2$ with the stated probability. \qed
\end{proof}

In this paper, we shall prove that it is possible to find an approximately feasible solution of a given SCP  (referred to as the {\it original} SCP), having an approximately optimal objective function value, by formulating another SCP (referred to as the {\it projected} SCP) subject to a random aggregation, obtained by means of a RP, of the original SCP constraints.

The rest of this paper is organized as follows. In Sect.~\ref{s:scp} we introduce FRJAs, the corresponding spectral theorem, and an abstract version of primal and dual SCP pairs. In Sect.~\ref{s:rpscp} we derive projected reformulations of primal and dual SCP using RPs, and prove approximation results for feasibility and optimality. In Sect.~\ref{s:retrieval} we discuss how to retrieve a solution of the original SCP from the projected SCP. In Sect.~\ref{s:compres} we present a computational validation of the foregoing theory.

\section{SCPs and Jordan Algebras}
\label{s:scp}
We consider the following SCP formulation in standard form:
\begin{equation}
  \left.\begin{array}{rrcl}
    \min & \transpose{c}x && \\
    &  Ax &=& b \\
    & x &\succeq_{\mathcal{K}}& 0,
  \end{array}\right\} \quad (\mathsf{SCP}) \label{scp}
\end{equation}
where $c,x\in\mathbb{R}^n$, $A\in\mathbb{R}^{m\times n}$, $b\in\mathbb{R}^m$, and $x\succeq_{\mathcal{K}} y$ means $x-y\in \mathcal{K}$, where $\mathcal{K}$ is a closed convex pointed cone with non-empty interior \cite{vandenberghe}. In general, $\mathcal{K}$ might be a cartesian product of linear, second-order, positive semidefinite (psd), and possibly other self-dual cones \cite{cvxopt}.

We are going to consider a generalization of (\textsf{SCP}) in the context of FRJAs \cite{faraut,jordanalg}. Using this approach, for example, complementarity and interior point algorithms can be described in terms of algebraic statements, including proofs of polynomial time convergence of interior point methods \cite{alizadeh}.

The formalization of conic programming in FRJA calls for an algebraic interpretation of symbols standing for vectors and matrices (as abstract elements and operators of an algebra). On the other hand, since an algebra {\it is} a vector space, the abstract elements can always be realized as vectors. Accordingly, we shall define a formal abstract inner product $\langle a,b\rangle$ over a FRJA, which may be realized as a quadratic form $\transpose{a}Qb$ for some square matrix $Q$ if we see $a,b$ as elements of the underlying vector space. Similarly, we shall denote $Ax$ by $A\odot x$ (after formally defining $\odot$ by means of the inner product), but then we shall argue that the $A\odot$ operation can be represented by a linear operator $\mathbb{A}$ on vectors, which can be realized by a matrix. 

\subsection{Formally real Jordan algebras}
\label{s:jordan}
We recall some basic facts about FRJAs \cite{jordanalg,faraut}. An algebra over $\mathbb{R}$ is a pair $(\mathbb{E},\circ)$ where $\mathbb{E}$ is a vector space over $\mathbb{R}$ and $\circ:\mathbb{E} \times \mathbb{E} \to \mathbb{E}$ is distributive in both arguments, i.e., for all $a,b,c\in\mathbb{E}$ and $\lambda,\gamma\in\mathbb{R}$, we have (i) $a \circ (\lambda b + \gamma c)  = \lambda (a \circ b) + \gamma (a \circ c)$ and (ii) $(\lambda b + \gamma c) \circ a  = \lambda (b \circ a) + \gamma (c \circ a)$. For each $x \in \mathbb{E}$, we let $L(x)$ be the {\it left multiplication operator}: $L(x)(y) = x \circ y$ for all $y\in\mathbb{E}$.

An algebra $(\mathbb{E},\circ)$ is called a {\it Jordan algebra} if it satisfies:%
\begin{enumerate}[(a)]
\item $\forall a,b\in \mathbb{E} \quad a \circ b = b \circ a$ \quad (commutativity);
\item $\forall a,b \in \mathbb{E}\quad a^2\circ(a\circ b)=a\circ (a^2\circ b)$ \quad (Jordan identity).
\end{enumerate}
The Jordan identity implies that a Jordan algebra is power-associative, i.e.~for all $p,q\in\mathbb{N}$ the matrices $L(x^p)$ and $L(x^q)$ commute. A Jordan algebra $(\mathbb{E},\circ)$ is called {\it formally real} if for any integer $m \ge 1$ and any finite set $\{x_1, \ldots, x_m\}\subset\mathbb{E}$ we have:
\begin{equation}
  \sum_{i\le m} x_i^2=0 \quad \Rightarrow \quad x_1 = \ldots = x_m =0.\label{eq:nonneg}
\end{equation}
Informally speaking, Eq.~\eqref{eq:nonneg} states that square terms in FRJAs must be non-negative. This is a nontrivial statement insofar as elements of $\mathbb{E}$ are not totally ordered in general (so one cannot write $x^2\ge 0$).

\subsection{The spectral decomposition theorem}
\label{s:sdt}
For each element $x\in\mathbb{E}$ the {\it degree} $\deg(x)$ of $x$ is the largest integer $d$ such that $e=x^0,x,\ldots,x^{d-1}$ (where $e$ is the multiplicative unit of $\mathbb{E}$) are linearly independent. The degree of $\mathbb{E}$, denoted by $\deg(\mathbb{E})$, is the maximum degree of all elements over $\mathbb{E}$. This allows us to state the spectral decomposition theorem for FRJAs \cite[Thm.~III.1.1]{faraut}.%
\begin{thm}
  Let $\mathbb{E}$ be a FRJA of degree $r$. For every $x\in\mathbb{E}$ there are sets \[\Lambda(x)=\{\lambda_i\;|\;i\le r\}\subset\mathbb{R}\quad\mbox{and}\quad C(x)=\{c_i\;|\; 1\le i\le r\} \subset\mathbb{E},\] with $\Lambda(x)$ ordered so that $\lambda_i\le\lambda_{i+1}$ for each $i<r$, and $C(x)$ satisfying the following properties:
  \begin{enumerate}[(i)]
  \item $\forall i\le r\;(c_i^2=c_i)$
  \item $c_1+\ldots+c_r=e$
  \item $\forall i\neq j\;(c_i\circ c_j=0)$, such that $x=\lambda_1c_1+\ldots+\lambda_rc_r$.
    \end{enumerate}
  \label{thm:sdt}
\end{thm}
The elements of $\Lambda(x)$ are the {\it eigenvalues} of $x$. Thm.~\ref{thm:sdt} is a generalization of the spectral decomposition theorem for real symmetric matrices. 

Let $\mathbb{E}$ be a FRJA of degree $r$, and $x \in \mathbb{E}$. We define $\mathsf{tr}(L(x))$ as the trace of the matrix representation of the left multiplication operator $L(x)$. We then define the associative bilinear form
\[T(x,y)=\mathsf{tr}(L(x \circ y))\]
(for $y\in\mathbb{E}$). Since $(\mathbb{E},\circ)$ is formally real, $T(x,y)$ defines a positive definite inner product denoted by $\langle x,y\rangle$. The element $x$ of the algebra is {\it positive semidefinite} (psd) if there is $y \in \mathbb{E}$ s.t.~$x=y^2$. The set $\mathcal{K}_{\mathbb{E}}$ of all psd elements of $\mathbb{E}$ is called {\it the cone of squares} of $\mathbb{E}$. We denote by $x\succeq 0$ the fact that $x$ is psd in $\mathbb{E}$. By Thm.~\ref{thm:sdt}, $x\succeq 0$ is psd iff all of its eigenvalues are nonnegative; its smallest eigenvalue is
\[\lambda_{\min}(x) = \min\limits_{u\neq 0} \frac{\langle u , x\circ u \rangle}{\|u\|^2_{\mathbb{E}}},\]
where $\|\cdot\|_{\mathbb{E}}$ is the norm induced by $\langle\cdot,\cdot\rangle$.

\subsection{Primal and dual SCPs in FRJA notation}
According to assumption (a) in Sect.~\ref{s:intro}, we introduce an explicit bound constraint $\langle e, x \rangle \le \theta$ in the primal SCP of the following primal/dual pair.
{%\small
\begin{center}
  \begin{minipage}{11cm}
    \begin{minipage}{5.5cm}
      \begin{equation*}
        \left.\begin{array}{rrcl}
          \min\limits_{x \in X} & \langle c,x\rangle && \\
          & A\odot x & = & b  \\
          & \langle e,x \rangle & \le & \theta \\
          & x &\succeq & 0
        \end{array} \right\} \quad (P)
      \end{equation*}
    \end{minipage}
    \hfill
    \begin{minipage}{5.5cm}
      \begin{equation*}
        \left. \begin{array}{rrcl}
          \max\limits_{y \in \mathbb{R}^m\atop\nu \ge 0} & \transpose{b} y -\theta \nu && \\
          \sum\limits_{i\le m} & y_iA_i - \nu e &\preceq & c,
        \end{array} \right\} \quad (D)
      \end{equation*}
      \vspace*{0.5em}
    \end{minipage}
  \end{minipage}
\end{center}
}%
where $A\odot x=(\langle A_i,x\rangle\;|\;i\le m), b\in\mathbb{R}^m$.

In the rest of the paper we assume that either $(P)$ is infeasible, or both $(P)$ and $(D)$ have optimal solutions $x^*$ and $y^*$ respectively. According to assumption (b) in Sect.~\ref{s:intro}, we also assume that strong duality holds. 

\section{Random projection of SCP constraints}
\label{s:rpscp}
Let $T=(T_{kj})\in\mathbb{R}^{d \times m}$ be a RP. We consider the projected primal/dual pair:
{%\small
\begin{center}
  \begin{minipage}{11cm}
    \begin{minipage}{5.5cm}
      \begin{equation*}
        \left. \begin{array}{rrcl}
          \min\limits_{x \in X} & \langle c, x\rangle && \\
          \forall k\le d & \langle\bar{A}_k, x\rangle &=& (Tb)_k \\
           & \langle e,x \rangle & \le & \theta \\
          & x &\succeq & 0.
        \end{array}\right\} \quad (P_T)
      \end{equation*}
    \end{minipage}
    \hspace*{0.5cm}
    \begin{minipage}{5.5cm}
      \begin{equation*}
        \left. \begin{array}{rrcl}
          \max\limits_{\dly\in\mathbb{R}^d\atop\nu \ge 0} &\transpose{(Tb)}\dly -\theta \nu &&\\
          \sum\limits_{k\le d} & \dly_k\bar{A}_k - \nu e&\preceq & c,
        \end{array} \right\} \quad (D_T)
      \end{equation*}
      \vspace*{0.3em}
    \end{minipage}
  \end{minipage}
\end{center}
}%
where $\bar{A}_k=\sum_{i\le m} T_{ki}A_i$ for every $k\le d$. Note that $(P_T)$ is a relaxation of $(P)$, since its constraints are obtained by aggregating the constraints of $(P)$. Hence $\mathsf{v}(P_T)\le\mathsf{v}(P)$. Moreover, for any feasible solution $(\dly^\prime,\nu^{\prime})$ of ($D_T$), we have:
{%\small
\begin{eqnarray*}
   \sum_{k\le d}\dly^{\prime}_k\bar{A}_k - \nu^{\prime} e &=& \sum_{k\le d}\dly^{\prime}_k\sum_{i\le m} T_{ki}A_i  - \nu^{\prime} e = \sum_{i\le m}(\transpose{T}\dly^{\prime})_iA_i  - \nu^{\prime} e\preceq c,
\end{eqnarray*}}
whence $(\transpose{T}\dly^{\prime},\nu^{\prime})$ is a feasible solution of the dual ($D$) of the original problem.

\subsection{Approximate feasibility}
Since $T$ acts linearly on left and right hand sides of the equality constraints of $(P)$, if $(P)$ is feasible then so is $(P_T)$. The issue is that $(P_T)$ may be feasible even if $(P)$ is not. We therefore assume that $(P)$ is infeasible, and prove that $(P_T)$ is infeasible whp.

We first claim that $(D)$ has non-empty interior: we can choose $y_i$ small enough and $\nu$ large enough so that $\sum_i y_i A_i  - v\nu \preceq c$ holds strictly.
%taking $\bar{\nu}\ge 0$ large enough so that $c+e\bar{\nu}\succ 0$ we can find a neighborhood $B(\bar{\nu})$ of $\bar{\nu}$ such that $\forall \nu \in B(\bar{\nu}),\ c+\nu e \succeq\eta e$ with $\eta >0$ ($\star$). Now we remark that $\{ \sum_{i\le m} y_iA_i \preceq \eta e \ |\ y \in \mathbb{R}^m  \}$ has non-empty interior: since $0\preceq\eta e$, we can always choose some small enough values of $y$ such that $\sum_i y_iA_i\preceq\eta e$. By ($\star$), this implies that $\sum_i y_iA_i - \nu e\preceq c$, as claimed.
Since $(P)$ is infeasible and $(D)$ has non empty interior, we deduce by duality that $(D)$ is unbounded. Hence there exists a dual feasible solution $(\hat{y},\hat{\nu})$ such that
\begin{equation}
\transpose{b} \hat{y} -\theta \hat{\nu}=1\qquad\sum_{i\le m} \hat{y}_iA_i - {e} \hat{\nu} \preceq 0. \label{eq:farkas}
\end{equation}
We shall prove that $(D_T)$ is also unbounded, by constructing $(\dly',\nu')$ such that $\transpose{(Tb)} \dly'-\theta \nu' >0$ and $\sum_{i\le m} (T^\top \dly')_iA_i - {e} \nu' \preceq 0$. Let $\dly'=T^\top \hat{y}$. By Prop.~\ref{scalarJLL}, whp we have that
\begin{equation}
  \label{eq:1}
  \transpose{(Tb)} \dly' =  \transpose{(Tb)} T\hat{y} \ge \transpose{b} \hat{y} - \varepsilon \|b\|_2\|\hat{y}\|_2.
\end{equation}
Let $N=\sum_{i\le m} \hat{y}_iA_i - {e} \hat{\nu}$ and $N'=\sum_{i\le m} (T^\top \dly')_iA_i-{e}\hat{\nu}$. While $N\preceq 0$ by feasibility of $(D)$, we derive an upper bound to $\lambda_{\max}(N')$  in terms of
\[\|\transpose{A}\!\|_{\Kinf}=\sup_{\|y\|_{\infty}=1} \|\transpose{A}y\|_{\mathbb{E}}.\]
\begin{lem}
  \label{lem:feas}
  Let $T \in \mathbb{R}^{d\times m}$ be a RP with $\varepsilon\in(0,1)$. For $(\hat{y},\hat{\nu})$ feasible ($D$) and $\dly'=T\hat{y}$, we have:
  \[\mathsf{Prob}\bigg[\lambda_{\max} (N') \le \varepsilon \|\hat{y}\|_2\|\transpose{A}\!\|_{\Kinf}  \bigg]\ge 1-2mce^{-\mathcal{C}\varepsilon^2 d}.\]
\end{lem}
\begin{proof} 
  Let $f$, $f'$ be the quadratic forms associated to $N$, $N'$, namely $f(u)=\frac{\langle u, N\circ u\rangle}{\|u\|^2_{\mathbb{E}}}$ and $f'(u)=\frac{\langle u,N'\circ u\rangle}{\|u\|^2_{\mathbb{E}}}$. In particular, replacing $\dly'$ with $T\hat{y}$, we have:
  {%\small
  \begin{equation*}
    \|u\|^2_{\mathbb{E}} f'(u) = \big\langle u, \big (\sum\limits_{i\le m}(\transpose{T}\dly')_i A_i - e\hat{\nu}\big) \circ u \big\rangle = \sum\limits_{i\le m}(\transpose{T} T \hat{y})_i \langle u, A_i\circ u \rangle - \|u\|^2_{\mathbb{E}}\hat{\nu}.
  \end{equation*}}%
Since $f(u)=\frac{\sum_{i} \hat{y}_i  \langle u, A_i\circ u \rangle}{\|u\|^2_{\mathbb{E}}} - \hat{\nu}$, we can write  $f'(u)=f(u) + \frac{\sum_{i} \hat{z}_i\langle u, A_i\circ u \rangle}{\|u\|^2_{\mathbb{E}}}$, where $\hat{z}=T^\top T \hat{y} - \hat{y} \in \mathbb{R}^m$. This yields
{%\small
\begin{equation*}
  f'(u)= f(u) + \frac{\langle u,\sum_{i} \hat{z}_i A_i\circ u \rangle}{\|u\|^2_{\mathbb{E}}} = f(u) + \frac{\langle u,(\transpose{A}\hat{z}) \circ u \rangle}{\|u\|^2_{\mathbb{E}}} \le  f(u) + \|\transpose{A}\!\|_{\Kinf} \|\hat{z}\|_\infty,
\end{equation*}}%
where the inequality on the right holds by definition of $\|\transpose{A}\!\|_{\Kinf}$. Note that $\lambda_{\max} (N')= \max_{u\neq 0} f'(u) $ and $\lambda_{\max} (N)= \max_{u\neq 0} f(u) \ge 0$ as $N\preceq 0$. Because of Cor.~\ref{cor:scjll} implies $\|\hat{z}\|_\infty \le \varepsilon \|\hat{y}\|_2$ whp, we deduce that 
\begin{equation*}
  \lambda_{\max} (N') \le \lambda_{\max} (N) + \|\transpose{A}\!\|_{\Kinf} \|\hat{z}\|_\infty  \le  \varepsilon \|\hat{y}\|_2 \|\transpose{A}\!\|_{\Kinf}.
\end{equation*}
with probability at least $1-2mce^{-\mathcal{C}\varepsilon^2 d}$, as claimed. \qed
\end{proof}

\begin{thm}
  \label{thm:infeasibility}
  Suppose $(P)$ infeasible and  $\varepsilon \|\hat{y}\|_2 (\|b\|_2+\|\transpose{A}\!\|_{\Kinf}) <1$. With the same notation as above, we have $\mathsf{Prob}[(P_T) \mbox{ is infeasible}]\ge 1-2(m+1)ce^{-\mathcal{C}\varepsilon^2 d}$. 
\end{thm}
\begin{proof}
  By Lemma \ref{lem:feas}, with probability at least $1-2mce^{-\mathcal{C}\varepsilon^2 d}$, we have $N' -  \varepsilon \|\hat{y}\|_2\|\transpose{A}\!\|_{\Kinf}e \preceq 0 $, whence $\sum_{i\le m} (T^\top\dly')_iA_i -   {e} (\hat{\nu}+\varepsilon \|\hat{y}\|_2\|\transpose{A}\!\|_{\Kinf})\preceq 0$. Moreover, by \eqref{eq:farkas} and \eqref{eq:1}, we obtain
  \[\mathsf{Prob}\big[\transpose{(Tb)} \dly' -\theta \hat{\nu} \ge 1 - \varepsilon \|b\|_2\|\hat{y}\|_2\big]\ge 2ce^{-\mathcal{C}\varepsilon^2 d}.\]
Let $\nu'=\hat{\nu}+\varepsilon \|\hat{y}\|_2\|\transpose{A}\!\|_{\Kinf}$. Then $\sum_{i\le m} (T^\top \dly')_iA_i - {e} \nu' \preceq 0$. Furthermore,
  \[\transpose{(Tb)} \dly'  -\theta \nu' = \transpose{(Tb)} \dly'  -\theta (\hat{\nu} +\varepsilon \|\hat{y}\|_2\|\transpose{A}\!\|_{\Kinf}) \ge 1 - \varepsilon \|\hat{y}\|_2 (\|b\|_2+\|\transpose{A}\!\|_{\Kinf}).\]
  Hence, if $\varepsilon \|\hat{y}\|_2 (\|b\|_2+\|\transpose{A}\!\|_{\Kinf})<1$, Eq.~\eqref{eq:1} is satisfied, so $(D_T)$ is unbounded, implying by duality that $(P_T)$ is infeasible with probability given by union bound arguments. \qed
\end{proof}

\subsection{Approximate optimality}
We show that any feasible solution $(\hat{y},\hat{\nu})$ of $(D)$ is close to a point of the form $(\transpose{T}\hat{\dly},\hat{\nu})$ whp, where $(\hat{\dly},\hat{\nu})$ is a feasible solution of the following family of relaxations of $(D_T)$ parametrized over a real number $\mu>0$:
{%\small
\begin{equation*}
  \left. \begin{array}{rrcl}
    \max\limits_{\dly \in \mathbb{R}^d\atop\nu \ge 0} & \transpose{(Tb)}\dly -\theta \nu && \\
     \sum\limits_{i\le m} & (\transpose{T}\dly)_i A_i -\nu e &\preceq& c + \mu e.
\end{array}
\right\} \quad (\tilde{D}_T^\mu)
\end{equation*}}%
We use ($D_T^\mu$) to show that the optimal objective function value of the original problem is approximately invariant w.r.t.~RPs.

Let $\dly'=T\hat{y} \in \mathbb{R}^d$. We claim that there exists $\mu>0$ s.t.~$(\dly',\hat{\nu})$ is a feasible solution of ($\tilde{D}^\mu_{T}$) whp. This will follow from $c+\mu e - \sum_{i\le m}(\transpose{T}\dly')_i A_i + e \hat{\nu}\succeq 0$ whp; or, equivalently, from $\lambda_{\min}(M')\ge-\mu$ for some suitable $\mu>0$ whp, where $M'=c-\sum_{i\le m}(\transpose{T}\dly')_i A_i + e \hat{\nu}$. Let $M=c -\sum_{i\le m} \hat{y}_iA_i + e \hat{\nu} $.  Since $(\hat{y},\hat{\nu})$ is a feasible dual solution of ($D$) we have $M\succeq 0$ and hence $\lambda_{\min}(M)\ge 0$. On the other hand $M'$ might fail to be psd, as $\lambda_{\min}(M')$ might be negative. The following result provides a lower bound to $\lambda_{\min}(M')$ in terms of $\|\transpose{A}\!\|_{\Kinf}$.
\begin{lem}
  \label{lem:eigenvalue}
  Let $T \in \mathbb{R}^{d\times m}$ be a RP with $\varepsilon\in(0,1)$. For $(\hat{y},\hat{\nu})$ feasible in ($D$) and $\dly'=T\hat{y}$, we have
  \[\mathsf{Prob}\big[\lambda_{\min} (M') \ge - \varepsilon \|\hat{y}\|_2\|\transpose{A}\!\|_{\Kinf}\big]\ge 1-2mce^{-\mathcal{C}\varepsilon^2 d}.\]
\end{lem}
%The proof of this lemma is very similar to that of Lemma \ref{lem:feas}, but we present it for completeness. \\
\begin{proof}
  Let $f$, $f'$ be the quadratic forms associated to $M$, $M'$, namely $f(u)=\frac{\langle u, M\circ u\rangle}{\|u\|^2_{\mathbb{E}}}$ and $f'(u)=\frac{\langle u,M'\circ u\rangle}{\|u\|^2_{\mathbb{E}}}$. By bilinearity, we have 
  \begin{align*}
   \|u\|^2_{\mathbb{E}} f'(u) &= \bigg\langle u, \big (c-\sum\limits_{i\le m}(\transpose{T}\dly')_i A_i + e \hat{\nu}\big) \circ u \bigg\rangle \\ & = \bigg\langle u, \big (c-\sum\limits_{i\le m}(\transpose{T} T\hat{y})_i A_i + e \hat{\nu}\big) \circ u \bigg\rangle \\
    &=  \langle u,  (c+e \hat{\nu})\circ u\rangle - \sum\limits_{i\le m}(\transpose{T} T \hat{y})_i \langle u, A_i\circ u \rangle.
  \end{align*}
Since $f(u)=\frac{\langle u,  (c+e \hat{\nu})\circ u \rangle}{\|u\|^2_{\mathbb{E}}}- \sum\limits_{i\le m} \hat{y}_i  \frac{\langle u, A_i\circ u \rangle}{\|u\|^2_{\mathbb{E}}}$,
we have
\begin{align*}
  f'(u) =  f(u) - \sum\limits_{i\le m} \big((\transpose{T} T \hat{y})_i - \hat{y}_i \big)\frac{\langle u, A_i\circ u \rangle}{\|u\|^2_{\mathbb{E}}}=f(u) - \sum\limits_{i\le m} \hat{z}_i\frac{\langle u, A_i\circ u \rangle}{\|u\|^2_{\mathbb{E}}},
\end{align*}
where $\hat{z}=T^\top T \hat{y} - \hat{y} \in \mathbb{R}^m$. This yields
{%\small
\begin{align*}
  f'(u) &= f(u) - \frac{\langle u,\sum_{i\le m} \hat{z}_i A_i\circ u \rangle}{\|u\|^2_{\mathbb{E}}} \\ & = f(u) - \frac{\langle u,(\transpose{A}\hat{z}) \circ u \rangle}{\|u\|^2_{\mathbb{E}}} \ge  f(u) - \|\transpose{A}\!\|_{\Kinf} \|\hat{z}\|_\infty 
\end{align*}
}%
where the inequality on the right holds by definition of $\|\transpose{A}\!\|_{\Kinf}$ (note also that the probability does not depend on $u$). Now we have: (i) $\lambda_{\min} (M')= \min \limits_{u\neq 0} f'(u) $; (ii) $\lambda_{\min} (M)= \min \limits_{u\neq 0} f(u)\ge 0$ since $M\succeq 0$; (iii) Cor.~\ref{cor:scjll} implies $\|\hat{z}\|_\infty \le \varepsilon \|\hat{y}\|_2$ whp. From (i)-(iii) we deduce that 
\begin{align*}
  \lambda_{\min} (M') & \ge \lambda_{\min} (M) - \|\transpose{A}\!\|_{\Kinf} \|\hat{z}\|_\infty  
  \ge  - \varepsilon \|\hat{y}\|_2 \|\transpose{A}\!\|_{\Kinf}.
\end{align*}
with probability at least $1-2mce^{-\mathcal{C}\varepsilon^2 d}$, as claimed. \qed
\end{proof}
We remark that we have a trivial upper bound $\|\transpose{A}\!\|_{\Kinf} \le \sum\limits_{i \le m} \rho(A_i)$,  where, for all $i\le m$, $\rho(A_i)=\max\limits_{u\neq 0} \frac{|\langle u,  A_i \circ u\rangle|}{\|u\|^2_{\mathbb{E}}}$ is the eigenvalue of $A_i$ having the largest absolute value.

By Lemma \ref{lem:eigenvalue}, we can let $\mu=\varepsilon \|\hat{y}\|_2\, \|\transpose{A}\!\|_{\Kinf}$ in ($D_T^\mu$), and obtain the relaxation:
{%\small
\begin{equation*}
  \left. \begin{array}{rrcl}
    \max\limits_{\dly\in\mathbb{R}^d\atop\nu \ge 0} & \transpose{(Tb)}\dly -\theta \nu && \\
    \sum\limits_{k\le d} & \dly_k\big(\sum\limits_{i\le m} T_{ki}A_i\big)- \nu e &\preceq& c + \varepsilon \|\hat{y}\|_2 \|\transpose{A}\!\|_{\Kinf} e
  \end{array}\right\} \quad (\tilde{D}_T)
\end{equation*}}%
of $(D_T)$. By construction, $(\dly',\hat{\nu})$ is a feasible solution of ($\tilde{D}_{T}$) whp, as claimed. 

\begin{lem}
  \label{lem:optvals}
  Let $(y^\ast,\nu^\ast)$ be an optimum of $(D)$. With the same notation as above, we have
  {%\small
  \begin{equation*}
    \mathsf{Prob}\big[\val{\tilde{D}_T}\ge\transpose{(Tb)} Ty^\ast - \theta \nu^\ast \ge \val{D} - \varepsilon\|b\|_2\|y^\ast\|_2\big]\ge 1-2(m+1)ce^{-\mathcal{C}\varepsilon^2 d}.
  \end{equation*}}
\end{lem}
\begin{proof}
By Prop.~\ref{scalarJLL}, we have $\transpose{(Tb)} Ty^\ast \ge \transpose{b} y^\ast - \varepsilon\|b\|_2\|y^\ast\|_2$ whp; since $\val{D}=\transpose{b}y^\ast -\theta\nu^\ast$, we obtain $\transpose{(Tb)} Ty^\ast - \theta \nu^\ast \ge \val{D} - \varepsilon\|b\|_2\|y^\ast\|_2$. On the other hand, $(y^\ast,\nu^\ast)$ is optimal in $(D)$ so it is also feasible in $(D)$, hence by Lem.~\ref{lem:eigenvalue} $(T y^\ast,\nu^\ast)$ is feasible in $(\tilde{D}_T)$ whp. Thus $\val{\tilde{D}_T}\ge\transpose{(Tb)} Ty^\ast - \theta \nu^\ast$, as $(\tilde{D}_T)$ is a maximization problem; the stated probability is given by union bound arguments. \qed
\end{proof}

We now consider the dual of ($\tilde{D}_T$):
{%\small
\begin{equation*}
  \left.              
  \begin{array}{rrcl}
    \min\limits_{x} & \big\langle c + \varepsilon \|\hat{y}\|_2 \|\transpose{A}\!\|_{\Kinf}e &,& x \big\rangle \\
    \forall k\le d & \langle\bar{A}_k, x\rangle &=& (Tb)_k \\
    & \langle e,x \rangle  &\le&  \theta \\
    & x &\succeq& 0.
  \end{array}
  \right\} \quad (\tilde{P}_T)
\end{equation*}}%
Let $x'$ be an optimal solution of $(P_T)$. By the weak duality theorem of conic programming, with probability at least $1-2mce^{-\mathcal{C}\varepsilon^2 d}$ we have:
{%\small
\begin{equation}\label{eq2}
\val{\tilde{D}_T} \le \val{\tilde{P}_T} \le \val{P_T} + \varepsilon \|y^*\|_2 \|\transpose{A}\!\|_{\Kinf}\langle e, x_T \rangle \le \val{P_T}+\varepsilon \|y^*\|_2 \|\transpose{A}\!\|_{\Kinf} \theta,
\end{equation}}%
where $(y^*,\nu^*)$ is an optimal solution of $(D)$.
Hence, by combining Eq. \eqref{eq2} with Lemma \ref{lem:optvals} and the fact that the strong duality theorem holds for the original problem $(P)$, we have proved the following result.
\begin{thm}
  \label{thm:opt}
 Assume strong duality between $(P)$ and $(D)$ holds. With the same notation as above, $\mathsf{Prob} \big[\val{P_T} \ge \val{P} - \varepsilon\|y^*\|_2 (\|\transpose{A}\!\|_{\Kinf} \theta+\|b\|_2)\big] \ge 1-(2m+1)ce^{-\mathcal{C}\varepsilon^2 d}$.
\end{thm}

In the rest of the paper, $d$ can be considered fixed to $\mathcal{C}_0\log(m)/\varepsilon^2$, where $\mathcal{C}_0$ is chosen so that the probability in Theorem \ref{thm:opt} is as high as desired.

\section{Solution retrieval}
\label{s:retrieval}
Solution retrieval, i.e.~obtaining a good solution $\tilde{x}$ for ($P$) from a solution $x_T$ of ($P_T$), is an important part of most RP-based techniques. The error analysis of $\xapprox$ w.r.t.~($P$) turns out to depend on the feasibility error of $x_T$ in ($P$).

\subsection{Feasibility error of the projected solution}
The \textit{Gaussian width} of a bounded set $S\subset\mathbb{R}^n$ is given by $w(S)=\mathbb{E}_g(\sup_{x \in S} g\cdot x)$, where $g$ is a standard normal vector in $\mathbb{R}^n$ and $\cdot$ is the usual scalar product on $\mathbb{R}^n$.  Here are some of its basic properties \cite{vershynin}:
\begin{itemize}
\item for any $v\in\mathbb{R}^n$, $w(S+y)=w(S)$;
\item $w(\mathsf{conv}(S))=w(S)$
\item for any $m \times n$ matrix $A$, $w(AS)\le \|A\|_2w(S)$, where $\|A\|_2$ is the matrix norm induced by the Euclidean norm;
\item if $S$ is finite, $w(S)\le \tilde{C} \sqrt{\log(|S|)} \mathsf{diam}(S)$ for some universal constant $\tilde{C}$.
\end{itemize}
Next, we recall the $M^*$ bound theorem, which bounds the diameter of the intersection of a subset $S \subset \mathbb{R}^n$ with $\ker(T)$, where $T$ is a RP \cite{vershynin}. We let $\mathsf{rad}(S)=\max_{x\in S} \|x\|_2$ and $\mathsf{diam}(S)=\max_{x,y\in S}\|x-y\|_2$.
\begin{thm}[$M^*$ bound theorem, \cite{vershynin}]
  \label{thm:mbound}
  For $S\subset \mathbb{R}^n$ and $u\ge 0$,
  \[\mathsf{Prob}\big[\mathsf{diam}(S \cap \ker(T)) \le \frac{C_2w(S)+u \mathsf{rad}(S)}{\sqrt{d}}\big]\ge1-2e^{-u^2},\]
  where $C_2$ is a universal constant.
\end{thm}

Let $\mathbb{A}$ be the matrix representation of the linear mapping $x \mapsto A\odot x$ from $\mathbb{R}^n$ to $\mathbb{R}^m$, i.e.~the $m \times n$ matrix such that for all $x \in \mathbb{R}^{{n}}$, $A\odot x= \mathbb{A}x$. Let $\mathbb{B}=\{x \in \mathcal{K}_{\mathbb{E}}\ |\ \langle e, x\rangle \le 1 \}$. Since $\mathcal{K}_{\mathbb{E}}$ is a pointed cone \cite[Ch.~I]{faraut}, $\mathbb{B}$ is a bounded region of $\mathbb{R}^n$ of given diameter $\Delta$ (as $\langle e,x\rangle>0$ over $\mathbb{K}_{\mathbb{E}}\smallsetminus\{0\}$), hence its Gaussian width $w(\mathbb{B})$ is finite.
%For example, when $\mathcal{K}_{\mathbb{E}}$ is the non-negative orthant $\mathbb{R}^n_+$ it is easy to see that $w(\mathbb{B}) \le \tilde{C} \sqrt{\log(n)}$ and $\Delta\le 2$.
\begin{prop}
  \label{sol:errorbound}
  Let $x_T$ be an optimal solution of $(P_T)$ and let $u >0$. Then we have:
  \[\mathsf{Prob}\big[\|\mathbb{A} x_T -b\|_2 \le \varepsilon \theta\|\mathbb{A}\|_2(C_2w(\mathbb{B})+u \Delta)/\sqrt{\log(n)}\big]\ge 1-2e^{-u^2}.\]
\end{prop}
\begin{proof}
  Consider $S=\{\mathbb{A}x-b \ |\ x \in \theta\mathbb{B}\}\subset \mathbb{R}^m$. By definition, $\mathbb{A}x_T-b\in S$; moreover, $\mathbb{A}x_T-b\in\ker(T)$ since $x_T$ is a solution of $(P_T)$, hence
  \[\mathbb{A}x_T-b\in S\cap\ker(T).\]
  Moreover, 
  \[w(S)=w(\theta \mathbb{A}\mathbb{B}-b)\le \theta \|\mathbb{A}\|_2w(\mathbb{B}).\]
  Furthermore, using the fact that there exists $x \in \mathbb{R}^n$ such that $\mathbb{A}x=b$, we have
  \[\mathsf{rad}(S)\le \|\mathbb{A}\|_2\mathsf{diam}(\theta \mathbb{B})\le \theta \|\mathbb{A}\|_2 \Delta.\]
  The result follows by Thm.~\ref{thm:mbound}. \qed
\end{proof}

\begin{comment}
In the LP case, for example, we have that $\mathcal{K}_{\mathbb{E}}=\mathbb{R}^n_+$ and we know that in that case $w(\mathbb{B}) \le \tilde{C} \sqrt{\log(n)}$ and $\Delta\le 2$, which give the following result.
\begin{cor}\label{bound:LP}
  If $\mathcal{K}_{\mathbb{E}}=\mathbb{R}^n_+$, then, with probability at least $1-2e^{-u^2}$,
  \[ \|\mathbb{A} x_T -b\|_2 = \|A x_T -b\|_2 \le \varepsilon \theta\|A\|_2\left( C_2 \tilde{C}+\frac{u}{\sqrt{\log(n)}}\right).\]
\end{cor}
\end{comment}

\subsection{Error analysis of the retrieved solution}
\label{s:erretr}
We define the retrieved solution $\xapprox$ of $(P)$ as the orthogonal projection of $x_T$ on the feasible set of $(P)$, i.e.~$\xapprox\in \argmin_{x} \{ \|x-x_T\|^2_2 \;|\; \mathbb{A}x=b \}$. By the KKT conditions, $\tilde{x}$ can be computed using the pseudoinverse of $\mathbb{A}$:
\begin{equation}
  \xapprox= x_T +\transpose{\mathbb{A}}(\mathbb{A}\transpose{\mathbb{A}})^{-1}(b-\mathbb{A} x_T). \label{eq:xretr}
\end{equation}
Since $\xapprox$ satisfies $\mathbb{A}x=b$ by construction, we only bound the error of $\xapprox$ w.r.t.~membership in the cone $\mathcal{K}_{\mathbb{E}}$ in terms of the negativity of $\lambda_{\min}(\xapprox)$.
\begin{lem}
For $x,a,b\in E$ such that $x=a+b$, $\lambda_{\min}(x)\ge\lambda_{\min}(a)+\lambda_{\min}(b)$.
\label{lem:eigsum}
\end{lem}
\begin{proof}
  For each $y\in\mathbb{E}$ we have $\lambda_{\min}(y) = \min\limits_{u\not=0}\frac{\langle u,y\circ u\rangle}{\|u\|_{\mathbb{E}}^2}$, so
  \vspace*{-0.5em}
  {%\small
  \begin{eqnarray*}
    \lambda_{\min}(x) &=& \min\limits_{u\not=0}\frac{\langle u,(a+b)\circ u\rangle}{\|u\|_{\mathbb{E}}^2} = \min\limits_{u\not=0}\bigg(\frac{\langle u,a\circ u\rangle}{\|u\|_{\mathbb{E}}^2} + \frac{\langle u,b\circ u\rangle}{\|u\|_{\mathbb{E}}^2}\bigg) \\ [-0.3em]
    &\ge& \min\limits_{u\not=0}\frac{\langle u,a\circ u\rangle}{\|u\|_{\mathbb{E}}^2} + \min\limits_{u\not=0} + \frac{\langle u,b\circ u\rangle}{\|u\|_{\mathbb{E}}^2} = \lambda_{\min}(a) + \lambda_{\min}(b).
  \end{eqnarray*}}
\end{proof}

\begin{thm}
  \label{thm:conerr}
  Suppose the inner product in $\mathbb{E}$ is realized by a bilinear form based on the symmetric matrix $Q$. With the notation above, we have: 
  \begin{equation}
    \mathsf{Prob}\left[\lambda_{\min}(\xapprox) \ge \lambda_1 - \varepsilon \theta \kappa(\mathbb{A})\|Q^{\frac{1}{2}}\|_2\frac{C_2w(\mathbb{B})+u \Delta}{\sqrt{\log(n)}}\right] \ge 1-2e^{-u^2}. \label{retr3}
  \end{equation}
\end{thm}
\begin{proof}
  Let $\sum\limits_{j\le r} \lambda_jc_j$ be the spectral decomposition of $x_T$. Since $x_T$ solves $(P_T)$, $x_T\in \mathcal{K}_{\mathbb{E}}$. So its eigenvalues are $0\le \lambda_1 \le \ldots \le \lambda_r$, where $r$ is the degree of $\mathbb{E}$. Suppose that the spectral decomposition of
  \[\tilde{x}-x_T=\transpose{\mathbb{A}}(\mathbb{A}\transpose{\mathbb{A}})^{-1}(b-\mathbb{A} x_T)\]
  is $\sum\limits_{j\le r}\mu_jc'_j$, with $\mu_1 \le \ldots \le \mu_r$. Then
  \begin{equation}
  \lambda_{\min}(\xapprox)\ge\lambda_1 + \mu_1\label{star}
  \end{equation}
by Lemma \ref{lem:eigsum}. We have $\|\tilde{x}-x_T\|_{\mathbb{E}}=\sqrt{\sum_{j\le r} \mu^2_j}$ by definition. Therefore
  \[|\mu_1|=\sqrt{\mu_1 ^2}\le\sqrt{\sum_j\mu_j^2}=\|\tilde{x}-x_T\|_{\mathbb{E}}.\]
  We are going to reason about $\|\tilde{x}-x_T\|_2$ instead, and then use the fact that $\langle a,b\rangle=\transpose{a}Qy$, which will allow us to derive a bound on $\|\tilde{x}-x_T\|_{\mathbb{E}}$ from a bound on $\|\tilde{x}-x_T\|_2$. By Eq.~\eqref{eq:xretr}, we have
  \[\|\tilde{x}-x_T\|_2\le  \|\transpose{\mathbb{A}}(\mathbb{A}\transpose{\mathbb{A}})^{-1}\|_2\|\mathbb{A} x_T -b\|_2.\]
  If $U\Sigma\transpose{V}$ is the Singular Value Decomposition (SVD) of $\mathbb{A}$, a straightforward computation yields $V\transpose{\Sigma}(\Sigma\transpose{\Sigma})^{-1}\transpose{U}$ as the SVD of $\transpose{\mathbb{A}}(\mathbb{A}\transpose{\mathbb{A}})^{-1}$, which implies that $\|\transpose{\mathbb{A}}(\mathbb{A}\transpose{\mathbb{A}})^{-1}\|_2 =\frac{1}{\sigma_{\min}}$ where $\sigma_{\min}$ is the smallest singular value of $\mathbb{A}$. Moreover, by Prop.~\ref{sol:errorbound}, we have
  \[\|\mathbb{A} x_T -b\|_2 \le \varepsilon \theta\|\mathbb{A}\|_2(C_2w(\mathbb{B})+u \Delta)/\sqrt{\log(n)}\]
  whp. Hence we obtain
  \begin{equation}
  \|\tilde{x}-x_T\|_2 \le \varepsilon \theta \kappa(\mathbb{A})(C_2w(\mathbb{B})+u \Delta)/\sqrt{\log(n)}, \label{Ssign}
  \end{equation}
  where $\kappa(\mathbb{A})=\frac{\|\mathbb{A}\|_2}{\sigma_{\min}}$ is the condition number of $\mathbb{A}$. Furthermore, if $\bar{z}$ is an upper bound to $\|z\|_2$, then
  \[\|z\|_{\mathbb{E}}=\langle z,z\rangle=\transpose{z}Qz=\transpose{(zQ^{\frac{1}{2}})}Q^{\frac{1}{2}}z=\|Q^{\frac{1}{2}}z\|_2\]
  (since the inner product in $\mathbb{E}$ is positive definite), whence $\|z\|_{\mathbb{E}}\le\|Q^{\frac{1}{2}}\|_2\|z\|_2$. Finally, by Eq.~\eqref{star} and \eqref{Ssign}, we get Eq.~\eqref{retr3} as claimed.\qed
\end{proof}

As regards optimality, $\langle c,\tilde{x}\rangle = \langle c,x_T\rangle + \langle c, \transpose{\mathbb{A}}(\mathbb{A}\transpose{\mathbb{A}})^{-1}(b-\mathbb{A}x_T)\rangle$, so by Cauchy-Schwartz we get
\[|\langle c,\tilde{x}\rangle -  \langle c,x_T\rangle| \le  \|c\|_2 \|\transpose{\mathbb{A}}(\mathbb{A}\transpose{\mathbb{A}})^{-1}\|_2\|b-\mathbb{A}x_T\|_2.\]
By the proof of Thm.~\ref{thm:conerr} $\|\transpose{\mathbb{A}}(\mathbb{A}\transpose{\mathbb{A}})^{-1}\|_2=\frac{1}{\sigma_{\min}}$, and we can bound $\|b-\mathbb{A}x_T\|_2$  by Prop.~\ref{sol:errorbound}, yielding (with the above notation):
\[ \mathsf{Prob} \left[ |\langle c,\tilde{x}\rangle -  \langle c,x_T\rangle| \le \varepsilon \theta\kappa(\mathbb{A})\|c\|_2\frac{C_2w(\mathbb{B})+u \Delta}{\sqrt{\log(n)}} \right] \ge 1-2e^{-u^{2}}. \]

\section{Computational validation}
\label{s:compres}
All our tests were carried out on a core of a quad-core Intel Xeon 2.1GHz CPU of a 8-CPU system with 64GB RAM running CentOS Linux, using Julia 0.6.1 and the {\sc Mosek} 9.1.5 \cite{mosek9} SDP solver. The general structure of our algorithm is very simple:
\begin{enumerate}
\item solve the original SCP and reduce it to standard form;
\item sample a RP and project the equality constraints of the SCP
\item solve the projected SCP and retrieve its approximate solution
\item compare the original and retrieved solutions on errors and CPU time.
\end{enumerate}
We employed Achlioptas' RP matrices \cite{achlioptas} in the ``sparser'' variant of \cite{kane} with density $0.1$. 

In order to validate our approach, we need to establish that: (a) projected versions of infeasible instances do not turn out to be feasible excessively often; and (b) the optimality and feasibility errors of retrieved solutions are not too large. We consider a set of randomly generated SDP instances. The $\varepsilon$ parameter in Defn.~\ref{defn:rp} was set to $0.2$ for the feasible set, and to $0.13$ for the infeasible set. All random coefficients in the instance data were sampled from a uniform distribution on $[0,1]$. 

\subsection{Infeasible instances}
We generated two groups of 5 infeasible instances. The generation details (number $m$ of equality constraints, number $n$ of variables, number $d$ of equality constraints in the projected problem, density \textsf{dens} of generated matrices) are reported in the first five columns of Table \ref{t:infeasible}.
\begin{table}[!ht]
  {\small
\begin{center}
\begin{tabular}{|rrrr|rr|rr|}\hline
  $m$ & $n$ & $d$ & \textsf{dens} & $P$ & $P_T$ & cpu & cpu${}_T$ \\ \hline
   1000 &  820 & 716 & 0.5 & Infeas & 100\% & 5.44 & 2.92 \\
   1000 & 1275 & 763 & 0.5 & Infeas & 100\% & 6.19 & 4.73 \\
  \hline
\end{tabular}
\end{center}
}
\caption{Infeasible instances. See Table \ref{t3} for detailed results.}
\label{t:infeasible}
\end{table}
The results for the feasible set are presented in Table \ref{t:infeasible}. We report the solution status obtained by the solver on $P$ (always infeasible over all instances in all groups), the percentage of success in detecting infeasibility when solving $P_T$, and the CPU time taken to solve $P$ and $P_T$. We remark that, by ``solving $P_T$'', we mean: (i) sample the RP, (ii) project the problem, and (iii) actually solve it. The success score was perfect. Results worsened when increasing $\varepsilon$: with $\varepsilon=0.14$ the success rate was 80\%; with $\varepsilon=0.2$ the success rate was 0\%. 

\subsection{Feasible instances}
We generated three groups of 10 feasible instances. As above, $m,n,d,\mathsf{dens}$ are reported in Table \ref{t:feasible}. The first and second groups both have identity cost matrices, but different generation densities; the third group has a random cost matrix. 
\begin{table}[!ht]
  {\footnotesize
\begin{center}
\begin{tabular}{|rrrr|rrrrr|rr|}\hline
 $m$ & $n$ & $d$ & \textsf{dens} & $\val{P}$ & $\val{P_T}$ & {\color{gray}err} & $\langle c,\xapprox\rangle$ & {\color{gray}err} & cpu & cpu${}_T$ \\ \hline
2000 & 1540 & 332 & 0.2 & 1540.630 & 461.289 & 0.7 & 1540.630 & 0.0 & 8.55 & 6.76 \\
2000 & 1540 & 332 & 0.5 & 1540.557 & 468.083 & 0.7 & 1540.557 & 0.0 & 17.28 & 8.31 \\
4000 & 1830 & 340 & 0.1 & 87.133 & -3126.177 & 1.0 & 87.133 & 0.0 & 23.68 & 11.11
\\ \hline
\end{tabular}
\end{center}
}
\caption{Feasible instances. See Table \ref{t4} for detailed results.}
\label{t:feasible}
\end{table}
The results for the feasible set are presented in Table \ref{t:feasible} (each line reports averages over the corresponding group). We report the optimal objective function value $\val{P}$ of the original problem, the optimal objective function value $\val{P_T}$ of the projected problem and the relative error
\[\frac{\val{P}-\val{P_T}}{\max(|\val{P}|,|\val{P_T}|)},\] the objective function value $\langle c,\xapprox\rangle$ of the retrieved solution $\xapprox$ and the relative error
\[\frac{\val{P}-\langle c,\xapprox\rangle}{\max(|\val{P}|,|\langle c,\xapprox\rangle|)},\]
the CPU time taken to solve $P$ and the CPU time taken to solve $P_T$. By ``solving $P_T$'' we mean: (i) sample the RP, (ii) project the problem, (iii) actually solve it, and (iv) retrieve the solution for $P$. Standard deviations on CPU time are between 1\% and 5\% of the averages.

In each case, our methodology was able to find the optimal solution of the original problem exactly, namely $\xapprox=x^\ast$, in roughly half the time. Accordingly, $\val{P}=\langle c,\xapprox\rangle$ for all instances; moreover, infeasibility and negative definiteness errors were always zero.

As is typical for RP-based methods, the ratio between projected and original CPU time decreases for increasing problem size: this is due to the fact that the number of constraints increases only logarithmically in the number of variables. Moreover, denser instances also gain a CPU advantage, due to the fact that denser problem take more time to solve, whereas projected problems are already typically denser than their original counterparts. The perfect score in solution quality is much less typical \cite{jllmor,ipco19}. While it certainly validates the approach, this type of success rate may well fail to hold for SDPs and other conic programs from real applications.

\begin{table}[!ht]
\begin{center}
{\small
\begin{tabular}{|r|rrrr|rr|rr|}\hline
  ID & $m$ & $n$ & $d$ & \textsf{dens} & $P$ & $P_T$ & cpu & cpu${}_T$ \\ \hline
  1 & 1000 &  820 & 716 & 0.5 & Infeas & Infeas & 5.39 & 2.94 \\
  2 & 1000 &  820 & 716 & 0.5 & Infeas & Infeas & 5.45 & 2.91 \\
  3 & 1000 &  820 & 716 & 0.5 & Infeas & Infeas & 5.43 & 2.90 \\
  4 & 1000 &  820 & 716 & 0.5 & Infeas & Infeas & 5.48 & 2.94 \\
  5 & 1000 &  820 & 716 & 0.5 & Infeas & Infeas & 5.46 & 2.91 \\ \hline
  1 & 1000 & 1275 & 763 & 0.5 & Infeas & Infeas & 6.06 & 4.70 \\
  2 & 1000 & 1275 & 763 & 0.5 & Infeas & Infeas & 6.11 & 4.76 \\
  3 & 1000 & 1275 & 763 & 0.5 & Infeas & Infeas & 6.15 & 4.70 \\
  4 & 1000 & 1275 & 763 & 0.5 & Infeas & Infeas & 6.53 & 4.80 \\
  5 & 1000 & 1275 & 763 & 0.5 & Infeas & Infeas & 6.09 & 4.71 \\
  \hline
\end{tabular}
}
\end{center}
\caption{Detailed results on infeasible instances.}
\label{t3}
\end{table}

\begin{table}[!ht]
\begin{center}
{\footnotesize
\begin{tabular}{|r|rrrr|rrr|rr|}\hline
ID & $m$ & $n$ & $d$ & \textsf{dens} & $\val{P}$ & $\val{P_T}$ & $\langle c,\xapprox\rangle$ & cpu & cpu${}_T$ \\ \hline
1 & 2000 & 1540 & 332 & 0.2 & 1562.983 & 479.373 & 1562.983 & 8.53 & 6.55 \\
2 & 2000 & 1540 & 332 & 0.2 & 1508.010 & 435.141 & 1508.010 & 8.47 & 6.79 \\
3 & 2000 & 1540 & 332 & 0.2 & 1512.362 & 467.159 & 1512.362 & 8.70 & 6.68 \\
4 & 2000 & 1540 & 332 & 0.2 & 1554.651 & 482.174 & 1554.651 & 8.61 & 6.76 \\
5 & 2000 & 1540 & 332 & 0.2 & 1536.599 & 451.946 & 1536.599 & 8.59 & 6.79 \\
6 & 2000 & 1540 & 332 & 0.2 & 1570.673 & 475.775 & 1570.673 & 8.53 & 6.77 \\
7 & 2000 & 1540 & 332 & 0.2 & 1556.245 & 473.629 & 1556.245 & 8.52 & 6.70 \\
8 & 2000 & 1540 & 332 & 0.2 & 1505.491 & 463.241 & 1505.491 & 8.52 & 6.87 \\
9 & 2000 & 1540 & 332 & 0.2 & 1564.126 & 463.902 & 1564.126 & 8.51 & 6.94 \\
10 & 2000 & 1540 & 332 & 0.2 & 1535.160 & 420.553 & 1535.160 & 8.55 & 6.70 \\ \hline
1 & 2000 & 1540 & 332 & 0.5 & 1547.108 & 449.995 & 1547.108 & 16.84 & 8.16 \\
2 & 2000 & 1540 & 332 & 0.5 & 1535.504 & 458.891 & 1535.504 & 17.12 & 8.43 \\
3 & 2000 & 1540 & 332 & 0.5 & 1514.214 & 478.169 & 1514.214 & 17.45 & 8.02 \\
4 & 2000 & 1540 & 332 & 0.5 & 1571.305 & 457.359 & 1571.305 & 17.32 & 8.31 \\
5 & 2000 & 1540 & 332 & 0.5 & 1529.055 & 471.201 & 1529.055 & 17.37 & 8.28 \\
6 & 2000 & 1540 & 332 & 0.5 & 1548.015 & 492.389 & 1548.015 & 17.49 & 8.90 \\
7 & 2000 & 1540 & 332 & 0.5 & 1527.581 & 445.974 & 1527.581 & 17.24 & 8.09 \\
8 & 2000 & 1540 & 332 & 0.5 & 1518.795 & 445.112 & 1518.795 & 17.23 & 8.65 \\
9 & 2000 & 1540 & 332 & 0.5 & 1550.371 & 474.559 & 1550.371 & 17.35 & 8.05 \\
10 & 2000 & 1540 & 332 & 0.5 & 1563.617 & 507.185 & 1563.617 & 17.36 & 8.20 \\ \hline
1 & 4000 & 1830 & 340 & 0.1 & 45.489 & -3080.881 & 45.489 & 22.19 & 10.97 \\
2 & 4000 & 1830 & 340 & 0.1 & 128.385 & -3181.674 & 128.385 & 24.27 & 10.72 \\
3 & 4000 & 1830 & 340 & 0.1 & 100.892 & -3300.381 & 100.892 & 24.50 & 10.99 \\
4 & 4000 & 1830 & 340 & 0.1 & 88.365 & -2941.452 & 88.365 & 24.30 & 11.09 \\
5 & 4000 & 1830 & 340 & 0.1 & 155.619 & -3005.528 & 155.619 & 22.14 & 11.21 \\
6 & 4000 & 1830 & 340 & 0.1 & 77.195 & -3353.770 & 77.195 & 24.01 & 11.62 \\
7 & 4000 & 1830 & 340 & 0.1 & 38.441 & -3021.306 & 38.441 & 24.40 & 11.08 \\
8 & 4000 & 1830 & 340 & 0.1 & 30.648 & -3020.162 & 30.648 & 24.87 & 11.08 \\
9 & 4000 & 1830 & 340 & 0.1 & 135.078 & -2985.326 & 135.078 & 24.39 & 11.27 \\
10 & 4000 & 1830 & 340 & 0.1 & 71.217 & -3371.290 & 71.217 & 21.69 & 11.09 \\ \hline
\multicolumn{8}{|r|}{\it Average\ } & {\it 16.50} & \textit{\textbf{8.73}} \\ \hline
\end{tabular}
}
\end{center}
\caption{Detailed results on feasible instances.}
\label{t4}
\end{table}

%\section{Conclusion}
%We proposed a random projection technique for conic problems. From the solution of a projected problem we can retrieve a solution for the original problem more efficiently than by solving the original problem itself. 

\ifels
\bibliographystyle{plain}
\else
\bibliographystyle{plain}
\fi
%\bibliography{dr2}
\bibliography{dr2}

\end{document}